\documentclass{amsart}
\usepackage{scalerel,lipsum,url}
\usepackage{geometry, longtable}
\newcommand\marquis[3][5]{%
  \scalebox{#1}{$\scaleleftright{\{}{\parbox{#2}{\centering #3}}{\}}$}%
}

\newtheorem{Th}{Theorem}[section]
\newtheorem{Pn}[Th]{Proposition}

\newtheorem{Co}[Th]{Corollary}
\newtheorem{Ex}[Th]{Example}
\newtheorem{Le}[Th]{Lemma}

\theoremstyle{remark}
\newtheorem{De}[Th]{Definition}

\usepackage[all,2cell]{xy}
\usepackage{verbatim}
\usepackage{rotating}

\input amssym.def
\input amssym.tex

\usepackage{pst-node}
\usepackage{pst-coil}
\usepackage{pst-plot}
\usepackage{pst-grad}
\usepackage{pstricks}

\begin{document}

\title{Non-Abelian Simple Groups Act with Almost All Signatures}

\author{Mariela Carvacho}
\address{Universidad T\'ecnica Federico Santa Mar\'ia}
\email{marie.carvacho@gmail.com}
\thanks{The first author was partially supported by Anillo PIA ACT1415 and Proyecto Interno USM 116.12.2}

\author{Jennifer Paulhus}
\address{Grinnell College}
\email{paulhus@math.grinnell.edu}
\thanks{The second author was partially supported by a Frank and Roberta Furbush Scholarship from Grinnell College.}

\author{Tom Tucker}
\address{Colgate University (emeritus)}
\email{ttucker@colgate.edu}

\author{Aaron Wootton}
\address{The University of Portland}
\email{wootton@up.edu}

\date{\today}


\dedicatory{}

\commby{}

\begin{abstract}
The topological data of a group action on a compact Riemann  surface is often encoded using a tuple $(h;m_1,\dots ,m_s)$ called its signature. There are two easily verifiable arithmetic conditions on a tuple necessary for it to be a signature of some group action. In the following, we derive necessary and sufficient conditions on a group $G$ for when these arithmetic conditions are in fact sufficient to be a signature for all but finitely many tuples that satisfy them. As a consequence, we show that all non-Abelian finite simple groups exhibit this property. 
\end{abstract}

\maketitle

\section{Introduction}

The question of which groups can appear as the automorphism group of a compact Riemann surface of genus $g>1$ (equivalently which groups exist as the monodromy group of a transitive branched cover) is an old one, see for example \cite{Hur2}.  Specific knowledge of these automorphism groups and the corresponding monodromy has important applications in the study of the mapping class group (e.g.\ \cite{BroWoo1}, \cite{NakNak} and \cite{TyrWin}), inverse Galois theory \cite{invgalois}, Shimura varieties \cite{shimura}, and Jacobian varieties (e.g.\ \cite{paulhus} and \cite{rojasjacs}).

We say that a tuple $(h; m_1,\ldots, m_s)$ is the signature of a finite group $G$ acting on a compact Riemann surface $X$ of genus $g\geq 2$ if the quotient surface $X/G$ has genus $h$ and the quotient map $\pi \colon X\rightarrow X/G$ is branched over $s$ points with orders $m_1,\ldots, m_s$. We call $h$ the {\em orbit genus} and  the numbers $m_1,\ldots, m_s$ the {\em tail}  of the tuple. When the tail has length zero, we write $(h;-)$.

Riemann's Existence Theorem, see Theorem \ref{th-riem}, provides arithmetic and group theoretic conditions which are necessary and sufficient for a tuple $(h; m_1,\ldots, m_s)$ to be the signature of $G$ acting on $X$. Outside of a few well known degenerate cases, as outlined in Theorem \ref{th-pot}, a tuple $(h; m_1,\ldots, m_s)$  satisfies the necessary arithmetic conditions of Theorem \ref{th-riem} if and only if its tail consists only of non-trivial element orders from $G$. In particular, for a given $G$, all tuples which satisfy just the arithmetic conditions are known and easy to compute. We call such tuples {\it potential signatures}, and each group has an infinite number of corresponding potential signatures.

In contrast however, outside of some low genus examples, or special families of groups, see for example \cite{Bro1}, \cite{Har1} and the database provided in \cite{Breu}, the potential signatures which satisfy the group theoretic conditions of Theorem \ref{th-riem} are far less well understood. The issue is that computationally, testing the group theoretic condition is much more difficult than the arithmetic conditions, and so it is not currently possible to get the same complete information we have for potential signatures. Indeed, the problem of determining which groups act with which signatures and other related problems is part of a very active area of research, in no small part due to the tremendous advances in computational power over the last few decades, see for example \cite{BarIzq}, \cite{BujCirCond}, \cite{Cond}, \cite{MulSid}. 

Since finding potential signatures is easy, but verifying when they are actually signatures is hard, these observations lead to the following natural line of inquiry: are there any groups for which the difference between potential signatures and actual signatures is finite? If $G$ is a group which exhibits this property, we shall henceforth say that $G$ {\it acts with almost all signatures} or simply $G$ is AAS. 

The importance of AAS groups in the wider picture is the following. The  {\em genus spectrum} of a group $G$ is the (infinite) set of integers $\mathcal{S}(G)=\{g_1,g_2,\dots \}$ such that for each $g_i \in \mathcal{S}(G)$, there exists a compact Riemann surface of genus $g_i$ on which $G$ acts. In \cite{K}, Kulkarni showed that each $g_i \in \mathcal{S}(G)$ satisfies a simple congruence dependent on the order of $G$ and its exponent and moreover, there exists an integer $\Sigma_{G}$, called the {\em minimum stable genus of $G$} (see \cite{Wea}), such that any $g\geq \Sigma_{G}$ which satisfies this congruence lies in $\mathcal{S}(G)$. Our work expands on Kulkarni's work in the following sense. If $G$ is an AAS group, then there exists a $\sigma_{G} \in \mathcal{S}(G)$ such that for any $g\geq \sigma_{G}$ in the genus spectrum, $G$ acts with {\em all} potential signatures on a surface of genus $g$. That is, for any $g\geq \sigma_{G}$, satisfaction of the arithmetic conditions for $G$ to act on a surface of genus $g$ with signature $(h;m_1,\dots ,m_s)$ is sufficient for the existence of a group action with that signature. Thus for AAS groups, finding the signatures with which they can act for the applications cited above becomes much more straightforward.

In the following note, we produce necessary and sufficient conditions for when a group $G$ is AAS. As a consequence of these conditions, we are able to show that all finite non-Abelian simple groups are AAS. We also provide deeper analysis of the structure of AAS groups including an array of explicit examples, and we lay the groundwork for further inquiry into such groups.

\section{\label{sec-prelim}Preliminaries}

We start by recalling some basic facts about group actions on surfaces and fixing some notation and terminology as introduced in \cite{BozWoo}. For a fixed finite group $G$, we define:
\[
  \mathcal{O}(G) = \{ |x| : x \in G \setminus \langle e_G \rangle \} =\{ n_1,\dots ,n_r\}.
\]
which we call the \emph{order set} of $G$. For convenience we shall assume $n_1<n_2<\dots <n_r$. For an arbitrary group $G$, we let $e_G$ denote its identity and $[G,G]$ the commutator (or derived) subgroup of $G$, the subgroup generated by all elements of the form $[g_1,g_2]=g_1^{-1}g_2^{-1}g_1g_2$ for $g_1$ and $g_2 \in G$. 

The following modern adaptation of Riemann's Existence Theorem, \cite[Prop.\ 2.1]{Bro1}, provides necessary and sufficient conditions for the existence of a finite group $G$ acting with signature $(h; m_1,\ldots, m_s)$ on a compact Riemann surface $X$ of genus $g \geq 2$:

\begin{Th}
\label{th-riem}
 A finite group $G$ acts on a compact Riemann surface $S$ of genus $g \geq 2$ with signature $(h;m_{1},\dots ,m_{s})$ if and only if:

\begin{enumerate}
\item
The Riemann--Hurwitz formula is satisfied:  $$g =1+|G|(h-1) +\frac{|G|}{2} \sum_{j=1}^{s} \bigg( 1-\frac{1}{m_{j}} \bigg), $$ and

\item
there exists a vector $(a_1,b_1,\dots a_g,b_g,c_1,\dots ,c_s)$ of elements of $G$ of length $2g+s$ called a $(h;m_{1},\dots ,m_{s})$-generating vector for $G$ which satisfies the following properties:

\begin{enumerate}
\item $G=\langle a_{1},b_{1},a_{2},b_{2},\dots ,a_h,b_h,c_{1},\dots ,c_{s} \rangle$.
\item The order of $c_{j}$ is $m_{j}$ for $1\leq j\leq s$.
\item $\prod_{i=1}^{h} [a_{i} ,b_{i} ] \prod_{j=1}^{s} c_{j} =e_G$.
\end{enumerate}

\end{enumerate}
\end{Th}

\noindent 
If $G$ admits an $(h;m_{1},\dots ,m_{s})$-generating vector $(a_1,b_1,\dots a_g,b_g,c_1,\dots ,c_s)$, then $m_i\in \mathcal{O}(G)$ for each $m_i$ and thus we can rewrite the signature $(h;m_{1},\dots ,m_{s})$ as $(h; [n_1, t_1], [n_2, t_2], \ldots, [n_r, t_r])$, which we call an {\em abbreviated signature}, where the pair $[n_i,t_i]$ denotes $t_i$ copies of $n_i$ (note that $t_i\geq 0$ for all $i$).

\section{\label{sec-pot}The Space of Potential Signatures}

From Theorem \ref{th-riem}, we see there are two arithmetic conditions necessary for a tuple $(h; m_1,\ldots, m_s)$ to be the signature of some finite group $G$ acting on a compact Riemann surface $X$ of genus $g\geq 2$: (1)  $m_i\in \mathcal{O}(G)$ for each $i$ and (2) the Riemann-Hurwitz formula is satisfied. Accordingly, we define the following:

\begin{De}
Let $G$ be a finite group with order set  $\mathcal{O}(G) =\{ n_1,\dots ,n_r\}$.

\begin{enumerate}
\item
For arbitrary $t_1,\dots t_s$ we call $(h; [n_1, t_1], [n_2, t_2], \ldots, [n_r, t_r])$ a {\em potential signature} for $G$ if it satisfies the Riemann-Hurwitz formula for some $g\geq 2$. We denote the space of all potential signatures for a given $G$ by $\mathcal{P}_{G}$.

\item
When an $(h; [n_1, t_1], [n_2, t_2], \ldots, [n_r, t_r])$-generating vector for $G$ also exists, we call the signature an {\em actual signature} for $G$. We denote the space of all actual signatures by $\mathcal{A}_{G}$.

\item
If $(h; [n_1, t_1], [n_2, t_2], \ldots, [n_r, t_r]) \in \mathcal{P}_{G}$, but $(h; [n_1, t_1], [n_2, t_2], \ldots, [n_r, t_r]) \notin \mathcal{A}_{G}$, we say it is a {\em non-signature} for $G$.

\end{enumerate}
\end{De}

Clearly we have $\mathcal{A}_{G} \subseteq \mathcal{P}_{G}$. Our goal is to provide necessary and sufficient conditions for when the set of non-signatures, $\mathcal{P}_{G}-\mathcal{A}_{G}$, is finite, so we first start by describing explicitly the set $\mathcal{P}_{G}$.

\begin{Th}
\label{th-pot}
A tuple $(h; [n_1, t_1], [n_2, t_2], \ldots, [n_r, t_r])$ is in $\mathcal{P}_{G}$ for $n_i\in  \mathcal{O}(G)$  if and only  it is not one of the following signatures:

\begin{enumerate}
\item
$(0;-)$

\item
$(0;[n_i,1])$, $i\in \{1,\dots ,r\}$

\item
$(0;[n_i,2])$, $i\in \{1,\dots ,r\}$

\item
$(0;[n_i,1],[n_j,1])$, $i,j\in \{1,\dots ,r\}$

\item
$(0;[2,2],[n_2,1])$

\item
$(0;[2,1],[3,2])$

\item
$(0;[2,1],[3,1],[4,1])$

\item
$(0;[2,1],[3,1],[5,1])$

\item
$(0;[2,1],[4,2])$

\item
$(0;[3,3])$

\item
$(0;[2,1],[3,1],[6,1])$

\item
$(0;[2,4])$

\item
$(1;-)$

\item
  $(h; [n_1, t_1], [n_2, t_2], \ldots, [n_r, t_r])$ where $|G|$ is even  with $k$ the largest power of $2$ dividing $|G|$, such that  the sum of the $t_i$ for which $n_i$ is divisible by $k$  is odd.
\end{enumerate}

\end{Th}

\begin{proof}
Simple application of the Riemann-Hurwitz formula shows that these are the only signatures which don't satisfy the arithmetic conditions given in Theorem \ref{th-riem} to be the signature of $G$ acting on a surface of any genus $g\geq 2$. 
\end{proof}

We note that many of these signatures do occur as signatures of a group action -- just on genus $0$ or $1$, see for example Section 3.4 of \cite{Breu}. The only exceptions are $(0;[n_i,1])$, $(0;[n_i,1],[n_j,1])$ and $(h; [n_1, t_1], [n_2, t_2], \ldots, [n_r, t_r])$. For the first two of these signatures,  the Riemann-Hurwitz formula produces a negative value for $g$, and for the third, a non-integral value. However, in all three cases, it is impossible to construct a generating vector for any group with such a signature since it would not satisfy condition 2(c) of Theorem \ref{th-riem}. For the first two signatures, this is clear. For the signature  $(h; [n_1, t_1], [n_2, t_2], \ldots, [n_r, t_r])$, this follows from the fact that its Sylow $2$-subgroup must be cyclic, and hence in order to satisfy condition 2(c) of Theorem \ref{th-riem}, any generating vector must contain an even number of elements whose order is divisible by the largest power of $2$ dividing $|G|$. 

\section{\label{sec-cond} Necessary and Sufficient Conditions for a Group to Act with Almost all Signatures}

We are now ready to consider the problem of determining conditions for when a group is AAS. Before we prove the main result, we start with some initial observations about AAS groups.

\begin{Le}
\label{le-nec}
Suppose $G$ is AAS. Then:

\begin{enumerate}
\item
$[G,G]$ contains an element of order every $n_i \in \mathcal{O}(G)$.

\item
For each $n_i \in \mathcal{O}(G)$, $G$ is generated by elements of order $n_i$. 

\end{enumerate}

\end{Le}

\begin{proof}
In each case, it suffices to provide an infinite list of non-signatures for any $G$ which doesn't satisfy the given condition.

For the first case, suppose that the derived subgroup of $G$ does not contain an element of order $n_i$. By Theorem \ref{th-pot}, each tuple $(h;n_i)$ for $h=1,2,\dots$ is a potential signature for $G$. However, no such tuple can be an actual signature for $G$ since the relation 2(c) from Theorem \ref{th-riem} cannot possibly hold. 

For the second case, suppose that $G$ is not generated by elements of order $n_i$. By Theorem \ref{th-pot}, each tuple $(0;\underbrace{n_i,\dots ,n_i}_{t-\text{times}})$ for $t=4,5,\dots$ is a potential signature for $G$. However, no such tuple can be an actual signature for $G$ since condition 2(a) from Theorem \ref{th-riem} cannot possibly hold. 

\end{proof}

Notice that this lemma implies that all AAS groups must be non-Abelian, as $[G:G]$ is trivial when $G$ is Abelian.

\begin{Le}
\label{le-odd}
Suppose $G$ satisfies the two conditions of Lemma \ref{le-nec}. For each $n_i \in \mathcal{O}(G)$, there exist elements $g_1\dots g_{L_{i}}$  of order $n_i$ with $L_{i}$ odd such that $g_1\dots g_{L_{i}}=e$. 
\end{Le}

\begin{proof}
Suppose that $G$ satisfies the two conditions of Lemma \ref{le-nec} and for a given $n_i$, let $h_1,\dots ,h_k$ denote all elements of $G$ of order $n_i$. Let $H\leq G$ denote the subgroup generated by even length products of the elements $h_1,\dots ,h_k$. We first show that either $G=H$ or $H$ has index $2$ in $G$.

Let $O$ denote the set of elements which can be written as odd length products of $h_1,\dots ,h_k$. Then left multiplication by $h_1$ defines an injection from $O$ into $H$, so $|O|\leq |H|$. It follows that $|H|\geq |G|/2$, so either $G=H$ or $H$ has index $2$ in $G$.

In both cases, since $G/H$ is Abelian, the derived subgroup $[G,G]$ must be a subgroup of $H$. By assumption  $[G,G]$ contains an element of order $n_i$, which we can assume, without loss of generality, to be $h_1$. 

Since $h_1\in H$, it follows that $h_1$ can be written as an even product of the elements $h_1,\dots ,h_k$. Denote this product by $\mathfrak{P}$. But then $\mathfrak{P}h_1^{-1}$ is a product of an odd number of elements equal to the identity in $G$. Hence the result.

\end{proof}

In fact, the two necessary conditions in Lemma \ref{le-odd} are also sufficient conditions for when a group is AAS.

\begin{Th}
\label{Th-cond}
A group $G$ is AAS if and only if:

\begin{enumerate}
\item
The derived subgroup $[G,G]$ contains an element of order every $n_i \in \mathcal{O}(G)$.

\item
For each $n_i \in \mathcal{O}(G)$, $G$ is generated by elements of order $n_i$. 

\end{enumerate}
\end{Th}

\begin{proof}
The necessity of the conditions are proved in Lemma \ref{le-nec}. 

In order to prove the sufficiency of the conditions, we shall prove that if $G$ satisfies these conditions, then a non-signature satisfies the following:
\begin{enumerate}
\item
The orbit genus is bounded.

\item
The length of the tail is bounded.

\end{enumerate}
It follows from these facts that there are at most a finite number of non-signatures, hence proving the result. 

To prove the orbit genus of a non-signature is bounded, we introduce some notation. Let $g_1,\dots ,g_k$ denote a generating set for $G$ and let $\mathfrak{w}$ denote the integer such that any element $g\in [G,G]$ can be written as a product of at most $\mathfrak{w}$ commutators. We shall show that if $h\geq k+\mathfrak{w}$, then $(h; m_1,\dots ,m_s)$ is an actual signature for any tail $m_1,\dots ,m_s$ by constructing an explicit generating vector for $G$. 

 Let $x_1,\dots ,x_s \in [G,G]$ denote elements of orders $m_1,\dots ,m_s$ (note, such elements must exist by assumption). Now since $x_1,\dots x_s\in [G,G]$, so is the product $x=x_1\cdots x_s$, and its inverse $x^{-1}$. Therefore, $x^{-1} =\prod_{i=1}^{w} [c_i,d_i]$ for some $w\leq \mathfrak{w}$. Letting $e$ denote the identity element in $G$, since $h\geq k+\mathfrak{w} \geq k+w$, we can construct the vector: $$\left(g_1,e,g_2,e,\dots ,g_k,e,c_1,d_1,\dots ,c_w,d_w,\underbrace{e,\dots ,e}_{2(h-k-w) \text{ times}}, x_1,\dots x_s \right).$$ Cleary $G$ is generated by this set of group elements, and we have $$\left( \prod_{i=1}^{k} [g_i,e] \right) \left( \prod_{i=1}^{w} [c_i,d_i] \right)\left( \prod_{i=1}^{h-k-w} [e,e]\right) \left( \prod_{i=1}^{s} x_i\right)= \left( \prod_{i=1}^{w} [c_i,d_i] \right) x=x^{-1}x=e,$$ so in particular, this vector is a generating vector for $G$. It follows that the orbit genus of a non-signature for $G$ is bounded. 
 
Next we show that the tail of a non-signature is bounded. We first make some simple observations and fix some notation.

First, for a given $n_i$, since we are assuming $G$ is generated by elements of order $n_i$, it follows that there exists an even integer $M_i$ such that $(0;[n_i,M_i]) \in \mathcal{A}_{G}$. Specifically, though there may be similar signatures with a shorter tail, at worst we can construct a generating vector with tail of even length by using a generating set of elements of order $n_i$ followed by their inverses. 

Second, by Lemma \ref{le-odd}, there is an odd integer $N_i$ such that $(0;[n_i,N_i]) \in \mathcal{A}_{G}$. As above, though there may be similar signatures with a shorter tail,  by concatenating the generating vector constructed in Lemma  \ref{le-odd} with the generating vector with signature  $(0;[n_i,M_i])$ described above, we can take $N_i=M_i+L_i$. 

For the fixed set of such integers $\{N_1,\dots ,N_r,M_1,\dots M_r\}$, define $$N={\rm Max} \{N_1,\dots ,N_r,M_1,\dots M_r\}.$$ 
 
 Now, for a given $N_i$, let $(g_1,\dots, g_{N_{i}})$ be a $(0;[n_i,N_i])$-generating vector for $G$ (and similarly for $M_i$). Since the elements of each of these generating vectors generate $G$, every element in $G$ can be written as a word in these elements (and their inverses). Therefore, we can define 
\begin{center} $\alpha_{N_i}=$%
\marquis[1]{4.5in}{Minimum word length such that every $g\in G$ can be written as a word in  $g_1,\dots, g_{N_{i}}$ of length less than or equal to $\alpha_{N_{i}}$}
 \end{center}
 \noindent
 and we similarly define $\alpha_{M_{i}}$. Finally, we define $$\alpha ={\rm Max} \{\alpha_{N_1},\dots ,\alpha_{N_r},\alpha_{M_1},\dots \alpha_{M_r}\}.$$ We shall show that any potential signature whose tail length $s$ satisfies $s>r(N+\alpha)$ is always an actual signature by constructing an explicit generating vector, and hence the tail of a non-signature is bounded. 
 
Fix a potential signature $(h;m_1,\dots m_s)=(h; [n_1, t_1], [n_2, t_2], \ldots, [n_r, t_r])$, with $s>r(N+\alpha)$. Since $s>r(N+\alpha)$, there exists an $n_i$ such that the corresponding $t_i> {\rm Max} (N_i+\alpha_{N_{i}}, M_i+\alpha_{M_{i}})$.  Without loss of generality, assume that $t_r> {\rm Max} (N_r+\alpha_{N_{r}}, M_r+\alpha_{M_{r}})$ and rewrite the signature $(h;m_1,\dots m_s)$ as $(h;m_1,\dots m_l,[n_r,t_r])$.  Let $x_1,\dots ,x_l \in G$ denote elements of orders $m_1,\dots ,m_l$ respectively, and let $x=x_1\cdots x_l$. Since $g_1,\dots ,g_{N_{r}}$ generate $G$, $x^{-1}$ can be written as a word $\mathbb{W}=a_1\cdots a_m$ in $g_1,\dots ,g_{N_{r}}$ where $m<\alpha$. 

Letting $y$ denote an arbitrary element of order $n_r$ in $G$, if $s-l-m$ is odd, we can construct the vector: $$\left(\underbrace{e,\dots ,e}_{2h \text{ times}},  x_1,\dots ,x_{l}, a_1,\dots ,a_m, \underbrace{y,y^{-1},\dots ,y,y^{-1}}_{(s-l-m-N_r)/2 \text{ times}}, g_1,\dots g_{N_r} \right)$$ and if $s-l-m$ is even, we can construct the vector $$\left(\underbrace{e,\dots ,e}_{2h \text{ times}},  x_1,\dots ,x_{l}, a_1,\dots ,a_m, \underbrace{y,y^{-1},\dots ,y,y^{-1}}_{(s-l-m-M_r)/2 \text{ times}}, g_1,\dots g_{M_r} \right).$$ In each case, clearly $G$ is generated by this set of group elements, and by construction, the relation in 2(c) from Theorem \ref{th-riem} is satisfied by both. In particular, this vector is a generating vector for $G$, so it follows that the length of a tail of a non-signature for $G$ is bounded. 

\end{proof}

Though we shall consider other families of groups which are AAS in more detail in the next section, Theorem  \ref{Th-cond} immediately implies that there is one special family of groups that is always AAS.

\begin{Th}
\label{cor-simple}
A non-Abelian finite simple group is AAS.
\end{Th}

\begin{proof}
Since $G$ is simple, its commutator subgroup is all of $G$, hence condition (1) of Theorem \ref{Th-cond} is satisfied. For condition (2), observe that for a given $n_i$, if we let $H$ be the subgroup generated by a conjugacy class of elements of order $n_i$, then it is necessarily normal, and hence equal to $G$. 

\end{proof}

\section{\label{sec-fam}Families of AAS Groups}

Non-Abelian simple groups are just one family which satisfy the the two conditions given in Theorem \ref{Th-cond}. In this section, we look at AAS group more generally. We start with the following result which shows that AAS groups fall into two very different categories.

\begin{Pn}
\label{Pn-perfectp}
If a group $G$ is AAS, then it is either a non-Abelian $p$-group, or a perfect group.
\end{Pn}

\begin{proof}
Suppose that $G$ is AAS. If $G$ is a $p$-group, then condition $(1)$ of Theorem \ref{Th-cond} ensures that $G$ is not Abelian. 

Now suppose that $G$ is not a $p$-group and let $p$ and $q$ be distinct primes dividing $|G|$. Since $G$ is AAS, by condition (2) of Theorem \ref{Th-cond}, $G$ must be generated by elements of order $p$. It follows that $G/[G,G]$ must also be generated by elements of order $p$ and so must be elementary Abelian of order $p^k$ for some $k$. By identical reasoning, since $G$ is also generated by elements of order $q$, $G/[G,G]$ must be elementary Abelian of order $q^s$ for some $s$. It follows that $G/[G,G]$ must be trivial, and in particular, $G$ is perfect. 
\end{proof}

Proposition \ref{Pn-perfectp} ensures that any AAS group is either a perfect group or non-Abelian $p$-group, so this leads to the obvious question of whether it is sufficient to be in one of these families. Unfortunately this is not the case. Though we shall see in the following that it is quite easy to generate families which are AAS, or are not AAS, determining a unified statement about such groups seems to be a very difficult problem. With this in mind, our goal here is to provide evidence, both computationally and through explicit examples of families, to motivate further study toward such a  unified statement.

\subsection{$p$-groups}

We start by looking at $p$-groups. The first few results illustrate both infinite families which are AAS and infinite families which are not. 

\begin{Pn}
\label{Pn-pgroupbigexp}
A non-Abelian $p$-group of order $p^n$ and exponent $p^{n-1}$ is never AAS. 
\end{Pn}

\begin{proof}
Suppose $G$ is a $p$-group of order $p^n$. Since $p$-groups have a normal subgroup of every possible order, let $H$ be a normal subgroup of $G$ of order $p^{n-2}$. Then the quotient $G/H$ is Abelian so the derived subgroup $[G,G]$ is a subgroup of $H$. However, since $H$ has order $p^{n-2}$, it cannot contain an element of order $p^{n-1}$, so neither does the derived subgroup $[G,G]$. Hence $G$ fails condition (1) of  Theorem \ref{Th-cond}.

\end{proof}

From this Proposition, we can exhibit  a larger family  (that includes the family of Theorem 5.2)  of non-Abelian $p$-groups which are not AAS.

\begin{Co}
\label{Co-pgroupex}
No metacyclic $p$-groups are AAS. 
\end{Co}

\begin{proof}
Suppose that $G$ is an AAS metacyclic $p$-group of order $p^n$ with normal cyclic subgroup $C$ and cyclic quotient $G/C$.  Now by assumption, $G$ is generated by elements of order $p$, so it follows that $G/C$ is also generated by elements of order $p$. Therefore, $G/C$ is cyclic of order $p$ and $C$ has index $p$ in $G$. It follows that either $G$ has exponent $p^n$ or $p^{n-1}$. In the first case, $G$ is Abelian, so cannot be AAS, and in the second case, Proposition \ref{Pn-pgroupbigexp} shows $G$ cannot be AAS. 
\end{proof}

Proposition \ref{Pn-pgroupbigexp} and Corollary \ref{Co-pgroupex} illustrate specific examples of groups which are not AAS. It is also fairly straightforward to construct explicit examples of AAS groups. Since any non-Abelian $p$-group of exponent $p$ must, by definition, be generated by elements of order $p$, and has a non-trivial commutator containing at least one element of order $p$, we can conclude that 

\begin{Pn}
\label{Pn-pgroupexp}
A non-Abelian $p$-group of exponent $p$ is AAS. 
\end{Pn}

The following useful result shows that once we find $p$-groups that are AAS, it is easy to generate more. 

\begin{Th}
\label{Th-pgroups}
Let $G$ be a p-group of exponent $p^e$ which is AAS.  Let $H$ be any other p-group of exponent $p^d$ where $d\leq e$.  Suppose that $H$ is generated by elements of order $p$.  Then $G\times H$ is AAS.  
\end{Th}

\begin{proof}
First note that the commutator subgroup of $G\times H$ is the direct product of $[G,G]$ and $[H,H]$. In particular, since $[G,G]$ contains elements of all possible orders and the order of elements of $H$ are a subset of the orders of $G$, it follows that the commutator subgroup of $G\times H$ must contain elements of all possible orders.  Hence Condition (1) of Theorem \ref{Th-cond} is satisfied. 

Now suppose that $p^i\in \mathcal{O}(G)$. By assumption, $G$ is generated by elements of order $p^i$, so let $g_1,\dots ,g_k$ be such a set of generators. Also, by assumption, $H$ is generated by elements of order $p$, so let $h_1,\dots ,h_r$ be such a set of generating vectors. Then it follows that the set $(g_i,e_H)$, $1\leq i\leq k$ ($e_H$ the identity of $H$) and $(g_1,h_j)$, $1\leq j\leq r$, all of which have order $p^i$ generate $G\times H$. Hence condition (2) of Theorem \ref{Th-cond} is satisfied. 

\end{proof}

The following is immediate.

\begin{Co}
\label{cor-pgroups}
The product of two AAS $p$-groups is an AAS $p$-group.
\end{Co}

The importance of Theorem \ref{Th-pgroups} and Corollary \ref{cor-pgroups} is that once we have found a single AAS $p$-group of any exponent, we can generate infinitely many other AAS $p$-groups of that same exponent simply by forming direct products. This suggests that AAS groups should appear quite regularly as the order of the group increases. Indeed, for $p\neq 2$, there exists at least one AAS $p$-group of order $p^n$ for $n\geq 3$ constructed by taking a direct product of the non-Abelian $p$-group of order $p^3$ (AAS by Proposition \ref{Pn-pgroupexp}) with an appropriate elementary Abelian $p$-group.

Though the frequency with which they appear should increase as the group order increases, how it increases relative to the growth rate of $p$-groups is not clear from our current work. The available computational data on $p$-groups is limited as the orders of the groups quickly get beyond the capabilities of computer algebra systems so it is difficult to determine reasonable conjecture from computational data. Therefore, we venture that an interesting line of study would be of the asymptotic limit $$\lim_{n\rightarrow \infty} (\# \text{AAS $p$-groups of order $p^n$})/ (\#\text{non-Abelian $p$-groups of order $p^n$}).$$

In Magma \cite{magma}, we checked all $p$ groups for $p\in \{2,3,5,7\}$ up to exponent $9$ for $p=2$, and up to exponent $7$ for the odd primes, and found only a small handful of AAS examples not covered by Proposition \ref{Pn-pgroupexp}.

\subsection{Perfect Groups} As with $p$-groups, we shall illustrate specific examples of perfect groups which are and are not AAS. We already know from Theorem \ref{cor-simple} that all non-Abelian finite simple groups are a family of perfect groups which are AAS, but not all perfect groups are AAS. 

\begin{Pn}
\label{Pn-perfectgroupex}
$SL(2,q)$ for $q$ odd is not AAS.
\end{Pn}

\begin{proof}
Each such group is a perfect group. However, for each such group, there is a unique subgroup of order $2$, hence $SL(2,q)$ cannot be generated by elements of order $2$.
\end{proof}

Conversely, it is not necessary for a perfect AAS group to be simple.

\begin{Ex}
\label{Ex-perfectgroupex}
The two perfect groups of order 960, and the perfect group of order 1080 are all non-simple AAS groups. One of the groups of order 960 is the group of inner automorphisms of the wreath product of $C_2$ and $A_5$.  Up to order $2000$, the only other perfect non-simple groups with AAS are the four perfect groups of order 1344, and three of the eight perfect groups of order $1920$.
\end{Ex}

Moreover, as with $p$-groups,  once we have found a perfect AAS group, it is easy to construct more.

\begin{Le}
\label{Le-Easy}
Suppose that $G$ and $H$ are AAS, $g\in G$, $h\in H$, $n\in \mathcal{O}(G)$ and $m\in \mathcal{O}(H)$. Then there exists an integer $s$ such that $g$ can be written as a product of $s$ elements of order $n$ from $G$ and $h$  can be written as a product of $s$ elements of order $m$ from $H$. 

\end{Le}

\begin{proof}
Since $G$ is AAS, we know that $g$ can be written as a product of elements of order $n$ from $G$. Moreover, we can assume this product has even length, since otherwise we can concatenate with an odd length product equal to the identity which exists by Lemma \ref{le-odd}. Likewise  $h$ can be written as a product of even length of elements of order $m$ from $H$. Assuming the length of the product equaling $g$ is less than or equal to the length equaling $h$, we can  concatenate an appropriate number of products of pairs $g_ig_i^{-1}$ for $g_i$ of order $n$ until the lengths are the same. 
\end{proof}

\begin{Th}
\label{Th-perfectgroups}
If $G$, $H$ are AAS groups and $|G|$ and $|H|$ have the same primes in their factorization, then $G\times H$ is AAS.  
\end{Th}

\begin{proof}
Since direct products of perfect groups are perfect, we only need to check condition (2) of Theorem  \ref{Th-cond}.

We need to show that for any $k\in \mathcal{O}(G\times H)$, the elements of order $k$ generate $G\times H$. Now, if $k\in \mathcal{O}(G\times H)$, then $k={\rm lcm} (a,b)$ where $a$ is the order of an element of $G$ and $b$ is the order of an element in $H$. Let $(g,h)$ be an arbitrary element in $G\times H$  -- we shall show it can be written as a product of elements of order $k$. 

First assume both $a,b>1$. Then by Lemma \ref{Le-Easy}, there exists an integer $s$ such that $g$ can be written as a product $g_1\dots g_s$ of $s$ elements of order $a$, and $h$ can be written as a product $h_1\dots h_s$ of $s$ elements of order $b$. It follows that $(g,h)=\prod (g_i,h_i)$, hence $(g,h)$ can be written as a product of elements of order $k$. 

Now assume without loss of generality that $b=1$. It immediately follows that we must have $a=k$. Now if $p$ is any prime dividing $a$, then since $p|H$ by assumption and $H$ is AAS, it is generated by elements of order $p$. In particular, there exists an integer $s$ such that $g$ can be written as a product $g_1\dots g_s$ of $s$ elements of order $a=k$, and $h$ can be written as a product $h_1\dots h_s$ of $s$ elements of order $p$. It follows that $(g,h)=\prod (g_i,h_i)$, where each $(g_i,h_i)$ has order $k$, and in particular $(g,h)$ can be written as a product of elements of order $k$. 

\end{proof}

As with the analogous result for $p$-groups, the importance of Theorem \ref{Th-perfectgroups} is that once we have found a single AAS  perfect group, we can generate infinitely many other AAS perfect groups by forming direct products with itself, or with other AAS perfect groups with the same primes in its factorization. In particular, this immediately shows that there are many perfect groups which are not simple which are themselves AAS, as we can take as many direct products of a simple group with itself as we please. As with $p$-groups, it suggests that AAS groups should appear quite regularly as the order of the group increases. How this number increases relative to the growth rate of perfect groups is not clear from our current work. Therefore, as with $p$-groups, we venture that an interesting line of study would be a study of the asymptotic growth rate of AAS perfect groups.

Magma \cite{magma}  contains a database of all perfect groups up to order 50,000.  We tested each of these for AAS and found that there are $229$ perfect groups up to order 50,000, of which $26$ are simple, and $86$ are non-simple and AAS.

\bibliographystyle{amsplain}

\begin{thebibliography}{99}

\bibitem{BarIzq}
G. Bartolini and  M. Izquierdo, \textit{On the connectedness of the branch locus of the moduli space of Riemann surfaces of low genus}, Proc. Amer. Math. Soc. 140 (2012), no. 1, 35--45.


\bibitem{magma} {W. Bosma, J. Cannon, \and C. Playoust} \textit{The Magma algebra system. I. The user language}. J. Symbolic Comput. 24 (1997) 235--265. \url{http://magma.maths.usyd.edu.au}

\bibitem{BozWoo}
S. Bozlee and A. Wootton,  \textit{Asymptotic equivalence of group actions on surfaces and Riemann-Hurwitz solutions}. Arch. Math. (Basel) 102 (2014), no. 6, 565--573.

\bibitem{Breu}
T. Breuer, \textit{Characters and automorphism groups of compact Riemann surfaces}. Cambridge University Press (2001).

\bibitem{Bro1}
S. A. Broughton, \textit{Classifying finite group actions on surfaces of low genus}. J. Pure Appl. Algebra {\bf 69} (1990), 233--270.

\bibitem{BroWoo1} S. A. Broughton, A. Wootton, \textit{Finite abelian subgroups of the mapping class group}. Algebr. Geom. Topol. 7 (2007), 1651--1697
 
 \bibitem{BujCirCond}  E. Bujalance, F. J. Cirre, and M. D. E. Conder, \textit{On automorphism groups of Riemann double covers of Klein surfaces}.  J. Algebra 472 (2017), 146--171. 
 
\bibitem{Cond}
M. Conder, \textit{An update on Hurwitz groups}. Groups Complex. Cryptol. 2 (2010), no. 1, 35--49. 
 
 \bibitem{shimura}
 P. Frediani, A . Ghigi, and M. Penegini, \textit{Shimura varieties in the Torelli locus via Galois coverings}.  Int. Math. Res. Not. IMRN 2015, no. 20, 10595--10623. 
 
 \bibitem{invgalois}
M. D. Fried, H.  Völklein, \textit{The inverse Galois problem and rational points on moduli spaces}.  Math. Ann. 290 (1991), no. 4, 771--800. 
 
 \bibitem{TyrWin}
T. Ghaswala and  R. Winarski, \textit{Lifting homeomorphisms and cyclic branched covers of spheres}. Michigan Math. J. 66 (2017), no. 4, 885--890. 

 
\bibitem{Har1}
W. J. Harvey, \textit{Cyclic groups of automorphisms of a compact Riemann surface}. Quart. J. Math. 17 (1966) 86--97. 

\bibitem{Hur2}
A. Hurwitz,  \textit{\"Uber algebraische Gebilde mit eindeutigen Transformationen in sich}. Math. Ann. 41 (1892), no. 3, 408--442 

\bibitem{K}
R.S. Kulkarni, \textit{Symmetries of surfaces}. Topology {\bf 26}\,(2), 195--203 (1987).

\bibitem{MulSid}
J. M\"uller and  S. Siddhartha, \textit{A structured description of the genus spectrum of Abelian p-groups.} Glasg. Math. J. 61 (2019), no. 2, 381--423. 

\bibitem{NakNak}
G. Nakamura and T. Nakanishi, \textit{Generation of finite subgroups of the mapping class group of genus 2 surface by Dehn twists}. J. Pure Appl. Algebra 222 (2018), no. 11, 3585--3594


\bibitem{paulhus}
J. Paulhus, \textit{Decomposing Jacobians of curves with extra automorphisms}. Acta Arith. 132 (2008), no. 3, 231--244.

\bibitem{rojasjacs}
A. M. Rojas, \textit{Group actions on Jacobian varieties}. Rev. Mat. Iberoam. 23 (2007), no. 2, 397--420. 

\bibitem{Wea}
A. Weaver, \textit{Genus spectra for split metacyclic groups}. Glasg. Math. J. 43 (2001), no. 2, 209--218.

\end{thebibliography}

\end{document}